\documentclass{article}
\usepackage{graphicx} 
\usepackage{amsfonts,amsmath,amssymb,amsthm}
\title{The Law of Large Numbers for Time-inhomogeneous Markov Chains under General Conditions}
\author{Aaron Lau and Kouji Yano}
\date{}

\begin{document}
\theoremstyle{plain}
\newtheorem{Thm}{Theorem}[section]
\newtheorem{Lem}[Thm]{Lemma}
\newtheorem{Prop}[Thm]{Proposition}
\newtheorem{Cor}[Thm]{Corollary}
\newtheorem{Conj}[Thm]{Conjecture}
\theoremstyle{definition}
\newtheorem{Ass}[Thm]{Assumption}
\newtheorem{Def}[Thm]{Definition}
\newtheorem{Rem}[Thm]{Remark}
\newtheorem{Ex}[Thm]{Example}
\newtheorem{Proper}[Thm]{Property}
\numberwithin{equation}{section}
\maketitle

\begin{abstract}
The weak and strong laws of large numbers for time-inhomogeneous Markov chains are studied under general conditions. First, under Drift Condition and Contraction Condition in
total variation, we prove the weak law of large numbers. Then, assuming Drift Condition together with a time-inhomogeneous
Doeblin minorization, we develop a Nummelin-type splitting and obtain a strong
law of large numbers. Our results utilize the invariant measure family in the sense of Liu--Lu (2025), and extend the classical
Harris–ergodic LLN to the time-inhomogeneous setting.
\end{abstract}

\section{Introduction}
The Law of Large Numbers (LLN) is one of the central results in probability theory.
Beginning with the classical form for independent random variables (Borel, Kolmogorov),
it was later extended to Markov chains.  For finite-state regular Markov chains,
Kemeny--Snell~\cite{KemenySnell1976} showed that if $\pi_j$ denotes the limiting probability
of being in state $j$, then $\pi_j$ also equals the long-run empirical frequency of visits to~$j$;
the LLN therefore connects invariant measures with time averages.  Subsequent developments
established links between LLN for Markov chains and martingale convergence theorems
(see Hall--Heyde~\cite{HallHeyde1980}).

For general state-space \emph{homogeneous} Markov chains, the modern theory was developed
by Meyn--Tweedie~\cite{MeynTweedie2009}, utilizing Drift Conditions,
small sets, and Harris recurrence.  Their regeneration (or Nummelin splitting) method has since
become a canonical tool for proving limit theorems such as LLN and CLT in the presence of
dependence. Also see~\cite{Martin} for the standard theory of the Lyapunov function method to study the ergodicity of homogeneous Markov chains.

In many applications of current interest---reinforcement learning, adaptive stochastic control,
and stochastic optimization---the natural objects are no longer homogeneous Markov chains.
Their transition kernels may change with time, giving rise to \emph{time-inhomogeneous}
Markov processes.  In this setting, even basic ergodic principles are substantially more
delicate.  An invariant distribution is replaced by a \emph{family} of time-indexed invariant
measures $\{\mu_n\}_{n\in\mathbb Z}$ (See \eqref{eq:invariant_family}).  Only recently has a general ergodic theory for such
chains been developed.  Liu--Lu~\cite{LiuLu2025} established existence and uniqueness of
$\{\mu_n\}_{n\in\mathbb Z}$ and the exponential ergodicity under Drift and Contraction Conditions which is weaker than Doeblin's Condition. Note that Vassiliou~\cite{Vassiliou} obtained the Law of Large Numbers for time-inhomogeneous Markov systems on finite state spaces in the settings different from ours.

Let $(X_n)_{n\in\mathbb Z}$ be a time-inhomogeneous Markov chain satisfying Drift Condition and Doeblin's Condition, and let $\{\mu_n\}_{n\in\mathbb Z}$ denote its unique
invariant measure family.  Using a Nummelin-type splitting construction combined with a
maximal coupling argument exploiting the geometric forgetting rate of~\cite{LiuLu2025}, we
show that for \emph{any} initial state $x$ with $V(x)<\infty$ and any bounded measurable function $g:\mathcal{X}\to\mathbb{R}$,
\begin{equation*}
    \frac{1}{n}\sum_{k=0}^{n-1} g(X_k)-\frac{1}{n}\sum_{k=0}^{n-1} \mu_k(g)
    \;\xrightarrow[n \to \infty]{\ \text{a.s.}\ }\; 0 .
\end{equation*}
That is, the empirical averages of the chain converge almost surely to the averages of $g$
with respect to the evolving invariant measures.  This recovers the classical Harris-ergodic
LLN for homogeneous Markov chains as a special case, and extends the ``time average = space
average'' principle to evolving, nonstationary stochastic dynamics.

Consequently, the present results offer a theoretical foundation for Monte Carlo methods in nonstationary stochastic dynamics, where explicit invariant distributions are probably unavailable but long-time averages are of primary interest.

The organization of this paper is as follows. In Section 2, we introduce certain basic notions and earlier results about time-inhomogeneous Markov chains as preliminaries. In Section 3, we show the main results without expositing proofs. In Section 4, we prove the weak law of large numbers based on the work of Liu--Lu~\cite{LiuLu2025}, where Drift Condition and Contraction Condition are leveraged. Section 5 consists of three subsections. In Subsection 5.1, we construct a splitting chain of Nummelin type. In Subsection 5.2, we see the Drift Condition guarantees the Markov chain to return to small set with the probability of geometric tail. In Subsection 5.3, we build the strong law of large numbers based on Drift Condition and Doeblin's Condition.

\section{Preliminaries}

We recall basic definitions and notation for time–inhomogeneous Markov chains and recall some results of Liu--Lu~\cite{LiuLu2025}.

Let $\mathcal{X}$ be a Polish space equipped with its Borel $\sigma$–field $\mathcal{B}(\mathcal{X})$.
Denote by $\mathcal{P}(\mathcal{X})$ the set of all probability measures on $\mathcal{X}$. Let $P$ be a family of transition probability kernels. That is, 
for $m \le n$ and $x \in \mathcal{X}$, we have
\[
    P(m,x,n,\cdot)\in \mathcal{P}(\mathcal{X}),
\]
with
$$P(m,x,m,\cdot)= \epsilon_x,$$
where $\epsilon_x$ denotes the Dirac mass at $x$, 
and the Chapman--Kolmogorov identity:
$$P(m_1,x,m_3,A)=\int_{\mathcal{X}}P(m_1,x,m_2,dy)P(m_2,y,m_3,A),$$
for $ m_1\le m_2\le m_3, \forall
m_1,m_2,m_3\in \mathbb{Z}$, $x\in \mathcal{X}$ and $A \in \mathcal{B}(\mathcal{X})$.

We want to consider an inhomogeneous Markov chain $\{X_n\}_{n\in \mathbb{Z}_+}$ with the transition probability kernels $P$, that is,
\begin{equation}\label{MarkovChain}
\mathbb{P}(X_{n+1} \in A \mid X_j:j\le n)= P(n,X_n, n+1, A)\quad a.s.
\end{equation}
for $n\in \mathbb{Z}_+$ and $A \in \mathcal{B}(\mathcal{X})$.

Given a probability measure $\nu \in \mathcal{P}(\mathcal{X})$, we define the push-forward
of $\nu$ under the transition kernel from time $m$ to $n$ by
\begin{equation}
    P^*_{m,n}\,\nu(A)
    := \int_{\mathcal{X}} P(m,x,n,A)\,\nu(dx),
    \qquad A \in \mathcal{B}(\mathcal{X}).
    \label{eq:pushforward}
\end{equation}
The associated
transition semigroup acts by
\begin{equation}
    P_{m,n}f(x)
    := \int_{\mathcal{X}} f(y)\, P(m,x,n,dy).
    \label{eq:semigroup}
\end{equation}
for any bounded measurable function $f : \mathcal{X} \to \mathbb{R}$.

\begin{Def}[Liu--Lu~\cite{LiuLu2025}]
A family of probability measures
$\{\mu_n\}_{n \in \mathbb{Z}} \subset \mathcal{P}(\mathcal{X})$
is called an \emph{invariant measure family} for $P$ if
\begin{equation}
    P^*_{m,n}\,\mu_m = \mu_n,
    \qquad \text{for all } m \le n.
    \label{eq:invariant_family}
\end{equation}
\end{Def}

\medskip
Although we are primarily interested in the forward chain $(X_n)_{n\ge 0}$, we assume the family of kernels $P(n,x,m,\cdot)_{n\le m}$ is defined for all $n \in \mathbb{Z}$. This allows us to use the invariant measure family framework of Liu--Lu~\cite{LiuLu2025}, which is naturally formulated on $\mathbb{Z}$. The results then apply in particular to the restriction of the chain to $\mathbb{Z}_+$.
If such a family $\{\mu_n\}_{n\in \mathbb{Z}}$ exists, then there exists an inhomogeneous Markov chain $\{X_n\}_{n \in\mathbb{Z}_+}$ such that the distribution of $X_n$ is $\mu_n$ for all $n \in \mathbb{Z}_+$.

The supremum norm $\|\cdot\|_{\infty}$ is denoted by
\[
    \|\varphi\|_\infty := \sup_{x \in \mathcal{X}} |\varphi(x)|
\]
for a measurable function $\varphi : \mathcal{X} \to \mathbb{R}$.
The total variation norm of $\mu$ is denoted by
\begin{equation}
    \|\mu\|_{\mathrm{TV}}
    := \sup_{\|\varphi\|_\infty \le 1}
        \left| \int_{\mathcal{X}} \varphi(x)\,\mu(dx) \right|
     = \sup_{A \in \mathcal{B}(\mathcal{X})} \mu(A)-\inf_{A \in \mathcal{B}(\mathcal{X})} \mu(A)
    \label{eq:TV_norm}
\end{equation}
for a finite signed measure on $\mathcal{X}$.
The difference of two probability measures is always a finite signed measure,
hence the total variation distance
\[
    \|\nu_1 - \nu_2\|_{\mathrm{TV}}
\]
is well-defined on $\mathcal{P}(\mathcal{X})$.

\medskip

\begin{Ass}[Drift Condition]\label{DriftCondition} There exists a Borel measurable function $V:\mathcal{X}\to [0,\infty]$ such that
\begin{enumerate}
    \item[(i)] the set $\{x:V(x)<\infty\}$ is nonempty;
    \item[(ii)] there exist $0<\gamma<1$ and $C>0$ such that for any $x\in\mathcal{X}$ and $n\in\mathbb{Z}$,
    $$\int_{\mathcal{X}}V(y)P(n-1,x,n,dy)\leq \gamma V(x)+C.$$
\end{enumerate}

For any $R>0$, define the set $$\mathcal{C}_R:=\{(x,y):V(x)+V(y)\leq R\}.$$

\end{Ass}
\begin{Ass}[Contraction Condition]\label{ContractionCondition}  For any $R>0$, there exist constants $n_0=n_0(R)\in \mathbb{N}$ and $0<\delta=\delta(R)<1$ such that for any $n\in \mathbb{Z}$,
$$\underset{(x,y)\in\mathcal{C}_R}{\sup}\|P(n-n_0,x,n,\cdot)-P(n-n_0,y,n,\cdot)\|_{TV}\leq 2(1-\delta),$$
which is equivalent to
$$\left| \int_{\mathcal{X}}\varphi(z)P(n-n_0,x,n,dz)-\varphi(z)P(n-n_0,y,n,dz)\right|\leq 2(1-\delta)$$
uniformly over all measurable functions $\varphi:\mathcal{X}\to\mathbb{R}$ with $\|\varphi\|_{\infty}\leq 1$.\\
\end{Ass}

Denote the weighted supremum norm $\|\cdot\|$ with respect to the function $V:\mathcal{X}\to[0,\infty]$ by
$$\|\varphi\|:= \underset{x\in \mathcal{X}}{\sup}\frac{|\varphi(x)|}{1+V(x)}.$$

Liu--Lu\cite{LiuLu2025} showed the existence and uniqueness of the invariant measure family for $P$ and exponential ergodicity under the above assumptions.

\begin{Thm}[Liu--Lu~\cite{LiuLu2025}]\label{InvariantExistence} Suppose Assumption~\ref{DriftCondition} and Assumption~\ref{ContractionCondition} hold. Then there exists a unique sequence $\{\mu_n\}_{n\in \mathbb{Z}}$ of probability measures satisfying $$\int_{\mathcal{X}}V(x)\mu_n(dx)<+\infty$$ for any $n\in \mathbb{Z}$ such that
$$P^*_{m,n}\mu_m = \mu_n \text{  for any }  m\leq n.$$
Moreover, the following assertions hold:
\begin{enumerate}
    \item[\textbf{(E1)}] there exist constants $0<\alpha <1$ and $M>0$ such that
    $$\left\|\int_{\mathcal{X}}\varphi(y)P(n-m,\cdot,n,dy)-\int_{\mathcal{X}}\varphi(y)\mu_n(dy)\right\|\leq M\alpha^m \left\|\varphi-\int_{\mathcal{X}}\varphi(y)\mu_n(dy)\right\|$$
    for any $n\in \mathbb{Z},m\in\mathbb{N}$ and $x\in \mathcal{X}$.
    \item[\textbf{(E2)}] there exist constants $0<\alpha<1$ and $\tilde{M}>0$ such that
    $$\|P(n-m,x,n,\cdot)-\mu_n\|_{TV}\leq \tilde{M}\alpha^m(1+V(x))$$
    for any $n\in \mathbb{Z},m\in\mathbb{N}$ and $x\in \mathcal{X}.$
\end{enumerate}
\end{Thm}

\section{Main Results}

In this setion, we present the main results regarding the Weak Law of Large Numbers (WLLN) and the Strong Law of Large Numbers (SLLN).

We first state the WLLN under Drift and Contraction Conditions.

\begin{Thm}\label{WLLN} Suppose Assumptions~\ref{DriftCondition} and~\ref{ContractionCondition} hold. For the invariant measure family $(\mu_n)_{n \in \mathbb{Z}}$ given by Theorem~\ref{InvariantExistence}, assume that we have the uniform finite $V$-moment: 
$$\underset{n\in \mathbb{N}}{\sup}\int_{\mathcal{X}}V(x)\mu_n(dx)<\infty.$$
Let $(X_n)_{n\in \mathbb{Z}_+}$ denote the time-inhomogeneous Markov chain with transition probability $P$ such that $X_0 \sim \mu_0$ (Consequently, $X_n \sim \mu_n$ for all $n \in \mathbb{Z}_+$). Let $g$ be a bounded measurable function on $\mathcal{X}$. Then the following convergence in probability holds:
$$\frac{1}{n}\sum\limits^{n-1}_{k=0}g(X_k)-\frac{1}{n}\sum\limits^{n-1}_{k=0}\mu_k(g)\xrightarrow[n\to\infty]{\mathbb{P}}0.$$
\end{Thm}

\noindent The proof will be given in Section 4.

\medskip

In order to study the SLLN, we use Doeblin's Condition to do Nummelin's splitting coupling~\cite{Nummelin}. As a small set, we define, for $R>0$, $$\mathcal{C}(R):=\{x\in \mathcal{X}:V(x)\leq R\}.$$ 
\begin{Ass}[Doeblin's Condition]\label{DoeblinCondition} For any $R>0$, there exist $0< \beta<1$, and a probability measure $\nu$ such that for any $n \in \mathbb{Z}$ and $x\in \mathcal{C}(R)$ we have
$$P(n-1,x,n,\cdot)\geq \beta \nu(\cdot).$$
\end{Ass}

\medskip

We next state the SLLN under Drift and Doeblin's Condition.

\begin{Thm}\label{generalLLN}
    Suppose Assumptions~\ref{DriftCondition} and~\ref{DoeblinCondition} hold. Let $x\in \mathcal{X}$ be such that $V(x) <\infty$. Let $g:\mathcal{X}\to \mathbb{R}$ be a bounded measurable function. Then, for the time-inhomogeneous Markov chain $(X_n)_{n\in \mathbb{Z}_+}$ with $X_0\sim \epsilon_x$, we have
     $$    \frac{1}{n}\sum_{k=0}^{n-1} g(X_k)-\frac{1}{n}\sum_{k=0}^{n-1} \mu_k(g)
    \;\xrightarrow[n \to \infty]{\ \text{a.s.}\ }\; 0 .$$
\end{Thm}

\noindent The proof of Theorem~\ref{generalLLN} will be given in Section 5.

\medskip

Now we consider a simple example of time-dependent Markov chain satisfying the law of large numbers. 

\begin{Ex}
    Consider the transition probability on the space $\mathcal{X}=\{1,2\}$:
$$P_{n-1,n}=
\begin{pmatrix}
    P(n-1,1,n,1) & P(n-1,1,n,2)\\
    P(n-1,2,n,1) & P(n-1,2,n,2) \\
\end{pmatrix}=
\begin{pmatrix}
    1-a(n) & a(n)\\
    b(n) & 1-b(n) \\
\end{pmatrix}$$
where $$a(n)=\frac{1}{3}+\frac{\sin n}{6},\; b(n)=\frac{1}{4}+\frac{\cos n}{8},\quad n\in \mathbb{Z}.$$
We see that $\frac{1}{6}\leq a(n)\leq \frac{1}{2},\;\frac{1}{8}\leq b(n)\leq \frac{3}{8}$, thus Doeblin's Condition is satisfied with $\beta=\frac{1}{4}$ and $\nu=(1/2,1/2)$. The Drift Condition can be easily verified if we take $V:\mathcal{X} \to \mathbb{R},\;V(1)=V(2)=2$, for instance. So, the corresponding invariant measure family exists and we denote it by $(\mu_k)_{k\in\mathbb{Z}}$. From the exponential ergodicity of Theorem~\ref{InvariantExistence}, we obtain
\[
P_{-n,k}=P_{-n,-n+1}\,P_{-n+1,-n+2}\cdots P_{k-1,k} \xrightarrow[n\to\infty]{} 
\begin{pmatrix}
    \mu_k(1) & \mu_k(2) \\
    \mu_k(1)  & \mu_k(2)  \\
\end{pmatrix}
\]
for every $k\in\mathbb{Z}$. 
We consider the Markov chain $(X_k)_{k \in \mathbb{Z}_+}$ with initial value $X_0=x_0\in\{1,2\}$ and set the function $g(x)=x$. By Theorem~\ref{generalLLN}, we have the SLLN:
$$\frac{1}{n}\sum_{k=0}^{n-1}X_k - \frac{1}{n}\sum_{k=0}^{n-1}\left(\mu_k(1)+2\mu_k(2)\right)\xrightarrow[n\to \infty]{\text{a.s.}} 0.$$
In particular, if $ \frac{1}{n}\sum\limits_{k=0}^{n-1}\left(\mu_k(1)+2\mu_k(2)\right)$ converges to a constant, say $C$, we have
$$\frac{1}{n}\sum_{k=0}^{n-1}X_k \xrightarrow[n\to \infty]{\text{a.s.}} C.$$
However, we do not know whether the limit $C$ exists or not.
\end{Ex}

\section{Proof of WLLN}

We give the proof of Theorem~\ref{WLLN}.
\begin{proof}[Proof of Theorem~\ref{WLLN}] Let us compute the covariance of $g(X_i)$ and $g(X_j)$ for $i<j$. By  \textbf{(E2)}  of Theorem~\ref{InvariantExistence}, we have
\begin{align*}
    &\,\,\,\,\,\,\,\left|\mathrm{Cov}(g(X_i),g(X_j))\right| \\
    &=\left|  \mathbb{E}[g(X_i)g(X_j)]-\mathbb{E}[g(X_i)]\mathbb{E}[g(X_j)]\right| \\
    &=\left|\int_{\mathcal{X}}g(x)\left(\mathbb{E}[g(X_j)\mid X_i=x]-\mu_j(g)\right)\mathbb{P}(X_i\in dx)\right|\\
    &\leq\int_{\mathcal{X}}|g(x)|\|g\|_{\infty} \|P(i,x,j,\cdot)-\mu_j)\|_{TV}\mathbb{P}(X_i\in dx)\\
    &\leq \|g\|_{\infty}^2 \tilde{M}\alpha^{j-i}\int_{\mathcal{X}}(1+V(x))\mathbb{P}(X_i\in dx)\\
    &\leq C\alpha^{j-i},
\end{align*}
where $C:=\|g\|_{\infty}^2\,\tilde{M}\,\underset{n\in \mathbb{N}}{\sup}\int_{\mathcal{X}}(1+V(x))\mu_n(dx)<\infty$. Then we have
\begin{align*}
    \mathrm{Var}\left(\sum\limits_{k=0}^{n-1} g(X_k)\right)&=\sum_{i,j=0}^{n-1} \mathrm{Cov}(g(X_i),g(X_j))\\
    &\leq \sum_{k=0}^{n-1} \mathrm{Var}(g(X_k))+2 \sum_{0\leq i<j\leq n-1}\left|\mathrm{Cov}(g(X_i),g(X_j))\right|\\
    &\leq 2n \|g\|_{\infty}^2 + 2 \sum_{i=0}^{n-2} \sum_{j=i+1}^{n-1} C \alpha^{j-i}\\
    &\leq 2n \|g\|_{\infty}^2 +2 \sum_{i=0}^{n-1} \sum_{k=0}^{\infty} C \alpha^{k}\\
    &=2n \|g\|_{\infty}^2 +n \frac{2C\alpha}{1-\alpha}=O(n).
\end{align*}
So we obtain $Var\left(S_n(g)\right)= O(n)$. Finally by Chebyshev's inequality, for every $\varepsilon>0$, we have
\begin{align*}
    \mathbb{P}\left(\left|\frac{1}{n}\sum_{k=0}^{n-1}g(X_k)-\frac{1}{n}\sum_{k=0}^{n-1}\mu_k(g)\right|>\varepsilon\right) \leq \frac{\mathrm{Var}\left(\sum_{k=0}^{n-1}g(X_k)\right)}{n^2 \varepsilon^2}=O\left(\frac{1}{n}\right).
\end{align*}
Therefore, we obtain the desired convergence.
\end{proof}

\section{Proof of SLLN}

\subsection{Construction of splitting chains}

It is easy to see that Doeblin's Condition (Assumption~\ref{DoeblinCondition}) implies Contraction Condition (Assumption~\ref{ContractionCondition}). In fact, we introduce the probability measure $Q_x$ and $Q_y$ such that
\begin{align*}
    P(n-1,x,n,\cdot)&=\beta \nu(\cdot)+(1-\beta)Q_x\\
    P(n-1,y,n,\cdot)&=\beta \nu(\cdot)+(1-\beta)Q_y.
\end{align*}
Then we have
$$\|P(n-1,x,n,\cdot)-P(n-1,y,n,\cdot)\|_{TV}\leq (1-\beta)\|Q_x-Q_y\|_{TV}\leq 2(1-\beta),$$
which implies the Contraction Condition.

\begin{Rem}
    Contraction Condition may not imply Doeblin's Condition. A counterexample is given by $\mathcal{X}= \{1,2,3\}$, $\mathcal{C}(R)=\mathcal{X}$, and
    $$P(1,\cdot)=\frac{1}{2}(\epsilon_1 +\epsilon_2),\quad P(2,\cdot)=\frac{1}{2}(\epsilon_2 + \epsilon_3),\quad P(3,\cdot)=\frac{1}{2}(\epsilon_3 + \epsilon_1).$$
    Then $\|P(x,\cdot)-P(y,\cdot)\|_{TV}=1$ for $x\neq y$, so the contraction bound holds with $\delta= \frac{1}{2}$. However, no $\beta >0$ and probability $\nu$ can satisfy $P(x,\cdot) \ge \beta \nu(\cdot)$. See Bansaye--Cloez--Gabriel~\cite{Doeblin} for other results about Doeblin's Condition.
\end{Rem}

\medskip

From now on, we fix $ R$ such that $R>\frac{C}{(1-\gamma)^2}$ and $\mathcal{C}(R)\not= \emptyset$ in the sequel, and start to split the chain with \emph{Nummelin's splitting coupling} method~\cite{Nummelin}.
We first split the space $\mathcal{X}$ itself by writing $\check{\mathcal{X}}=\mathcal{X}\times \{0,1\}$, where $\mathcal{X}_0:=\mathcal{X}\times \{0\}$ and $\mathcal{X}_1:=\mathcal{X}\times\{1\}$ are thought of as copies of $\mathcal{X}$ equipped with copies $\mathcal{B}(\mathcal{X}_0)$, $\mathcal{B}(\mathcal{X}_1)$ of the $\sigma$-field $\mathcal{B}(\mathcal{X})$.
We let $\mathcal{B}(\check{\mathcal{X}})$ be the $\sigma$-field of subsets of $\check{X}$ generated by $\mathcal{B}(\mathcal{X}_0)$, $\mathcal{B}(\mathcal{X}_1)$: that is, $\mathcal{B}(\check{\mathcal{X}})$ is the smallest $\sigma$-field containing sets of the form $A_0:= A\times \{0\}$, $A_1:=A\times \{1\}$, $A\in \mathcal{B}(\mathcal{X})$.
We will write $x_i=(x,i), i=0,1$ for elements of $\check{\mathcal{X}}$, with $x_0$ denoting members of $\mathcal{X}_0$ and $x_1$ denoting the member of $\mathcal{X}_1$.
If $\lambda$ is any measure on $\mathcal{B}(\mathcal{X})$, then we split the measure $\lambda$ into two measures on each of $\mathcal{X}_0$ and $\mathcal{X}_1$ by defining the measure $\lambda^*$ on $\mathcal{B}(\check{\mathcal{X}})$ through
\begin{align*}
    \lambda^*(A_0) &= (1-\beta) \lambda(A \cap \mathcal{C}(R)) + \lambda(A \cap \mathcal{C}(R)^c)\\
    \lambda^*(A_1) &=\beta \lambda(A \cap \mathcal{C}(R)).
\end{align*}
It is critical to note that $\lambda$ is the marginal measure induced by $\lambda^*$, in the sense that for any $A$ in $\mathcal{B}(\mathcal{X})$ we have 
$$\lambda^*(A_0 \cup A_1)= \lambda(A).$$
We also note that
$$\lambda^*\left(\mathcal{C}(R)^c \times\{1\}\right)=0,\quad \lambda^*\left(\mathcal{C}(R) \times\{1\}\right)=\lambda(\mathcal{C}(R)).$$
Now we want to split the chain $(X_n)_{n\in \mathbb{N}}$ to form a chain $(\check{X}_n)_{n\in\mathbb{N}}$ where $\check{X}_n = (X_n, \delta_n)$, which lives on $(\check{\mathcal{X}},\mathcal{B}(\check{\mathcal{X}}))$ is a time-inhomogeneous Markov chain with the transition probabilities given by the split kernels defined by:
\begin{equation}\label{construction1}
    \check{P}(n,x_0,n+1,\cdot) = P(n,x,n+1,\cdot)^*,\quad\text{if }\; x_0\in\mathcal{X}_0-\mathcal{C}(R)_0;
\end{equation}
\begin{equation}\label{construction2}
    \check{P}(n,x_0,n+1,\cdot)=\frac{P(n,x,n+1,\cdot)^*-\beta\nu^*(\cdot)}{1-\beta},\quad\text{if }\;x_0\in\mathcal{C}(R)_0;
\end{equation}
\begin{equation}\label{construction3}
    \check{P}(n,x_1,n+1,\cdot)=\nu^*(\cdot),\quad\text{if }\;x_1\in\mathcal{C}(R)_1.
\end{equation}

Outside $\mathcal{C}(R)$ the chain $\{\check{X}_n\}$ behaves just like $\{X_n\}$, moving on $\mathcal{X}_0$ of the split space. Each time it arrives in $\mathcal{C}(R)$, it is split; with probability $1-\beta$ it remains in $\mathcal{C}(R)_0$, with probability $\beta$ it drops to $\mathcal{C}(R)_1$.
Now we denote the $\sigma$-fields: 
\begin{align*}
    \mathcal{G}_n&:=\sigma\left(X_0,X_1,\cdots,X_n,\delta_0,\delta_1,\cdots,\delta_{n-1}\right)\\
    \mathcal{F}_n&:=\sigma\left(X_0,X_1,\cdots,X_n,\delta_0,\delta_1,\cdots,\delta_{n-1},\delta_n\right).
\end{align*}

\begin{Prop}\label{Property1}
The split chain $\left\{\check{X}_n=(X_{n },\delta_n): n\in \mathbb{Z}_+\right\}$ on the space $\mathcal{X}\times\{0,1\}$ satisfies the following properties:
\begin{equation}\label{following1}
    \check{\mathbb{P}}\{\delta_n=1 \mid  \mathcal{G}_n;X_{n}=x\}=\left\{
    \begin{array}{cc}
        \beta, & \mathrm{if}\; x\in\mathcal{C}(R) \\
         0, & \mathrm{if}\;x\notin \mathcal{C}(R) 
    \end{array};
    \right. 
\end{equation}
\begin{equation}\label{following2}
    \check{\mathbb{P}}\{X_{n+1} \in dy\mid  \mathcal{F}_n;\delta_n=1\} = \nu(dy);
\end{equation}
\begin{equation}\label{following3}
    \check{\mathbb{P}}\{X_{n+1} \in dy\mid  \mathcal{F}_n;X_n=x,\delta_n=0\}= \left\{
    \begin{array}{cc}
        \frac{P(n,x,n+1,dy)-\beta \nu(dy)}{1-\beta}, & \mathrm{if}\;x\in \mathcal{C}(R) \\
         P(n,x,n+1,dy),  & \mathrm{if}\;x\notin \mathcal{C}(R) 
    \end{array}
    \right.
\end{equation}
Moreover, given that $\delta_{n}=1$, the pre-$n$ process $\{X_j,\delta_j:j\leq  n\}$ and post-$n$ $\{X_j , \delta_j: j\geq n+1\}$ process are independent.
\end{Prop}
\begin{proof}
    We construct the split chain explicitly expositing the underlying randomness. First, we introduce auxiliary i.i.d. sequences $(U^X_n)_{n\ge 0}$ and $(U^{\delta}_n)_{n\ge 0}$ of Uniform $(0,1)$ random variables, which are independent. According to Proposition 11.6 of Kallenberg~\cite{Kallenberg}, there exist measurable functions:
    \begin{align*}
        &\varphi^1_n:\mathcal{X}\times (0,1) \to \mathcal{X}\text{ s.t. }\varphi^1_n(x,U^X_n) \sim P(n,x,n+1, \cdot);\\
        &\varphi^2_n:\mathcal{X}\times(0,1) \to \mathcal{X}\text{ s.t. }\varphi^2_n(x, U^X_n) \sim \frac{P(n,x,n+1,\cdot)-\beta\nu(\cdot)}{1-\beta};\\
        &\varphi^3:\mathcal{X} \to \mathcal{X}\text{ s.t. } \varphi^3(U^X_n) \sim\nu .
    \end{align*}
    Define $\varphi_n:\mathcal{X}\times \{0,1\}\times(0,1) \to \mathcal{X}$ as
    \[
    \varphi_n(x,i,u):=\left\{
    \begin{array}{cc}
        \varphi^1_n(x,u), & \text{if }x\notin \mathcal{C}(R) \\
         \varphi^2_n(x,u),& \text{if }x\in\mathcal{C}(R), i=0 \\
         \varphi^3 (u), & \text{if }x\in \mathcal{C}(R), i=1
    \end{array}\right.
    \]
    Let $x^0$ be fixed and Construct $(X_n,\delta_n)_{n=0}^{\infty}$ as follows:
    \begin{itemize}
        \item $X_0=x^0$, 
        $\delta_0 = \left\{
        \begin{array}{cc}
             0,&\text{if } x\notin \mathcal{C}(R) \\
             \textbf{1}\{U^{\delta}_0 \leq \beta\},& \text{if }x \in \mathcal{C}(R) 
        \end{array}
        \right.
        $
        \item If $(X_n, \delta_n)$ is given, then\\ $X_{n+1}= \varphi_n\left(X_n,\delta_n,U^x_{n+1}\right)$, $\delta_{n+1} = \left\{
        \begin{array}{cc}
             0,& \text{if }X_{n+1}\notin \mathcal{C}(R) \\
             \textbf{1}\{U^{\delta}_{n+1} \leq \beta\},& \text{if }X_{n+1}\in \mathcal{C}(R) 
        \end{array}
        \right.$
    \end{itemize}
    Denote:
    \begin{align*}
        \tilde{\mathcal{G}}_n&:=\sigma\left(U^x_0,U^x_1,\cdots,U^X_n,U^{\delta}_0,U^{\delta}_1,\cdots,U^{\delta}_{n-1}\right) \supset \mathcal{G}_n ,\\
        \tilde{\mathcal{F}}_n&:=\sigma\left(U^x_0,U^x_1,\cdots,U^X_n,U^{\delta}_0,U^{\delta}_1,\cdots,U^{\delta}_{n}\right) \supset \mathcal{F}_n .
    \end{align*}
    By such construction, we have
    \begin{align*}
        &\check{\mathbb{P}}\left(\delta_n=1 \mid \tilde{\mathcal{G}}_n,X_n=x\right)\\
        &=\check{\mathbb{P}}\left(\delta_n=1 , X_n \notin \mathcal{C}(R)\mid \tilde{\mathcal{G}}_n,X_n=x\right)+\check{\mathbb{P}}\left(\delta_n=1 ,X_n\notin \mathcal{C}(R)\mid \tilde{\mathcal{G}}_n,X_n=x\right)\\
        &=0+\check{\mathbb{P}}\left(U^{\delta}_n \le \beta \mid \tilde{\mathcal{G}}_n,X_n=x\right)\textbf{1}\{x\in \mathcal{C}(R)\}\\
        &=\beta\textbf{1}\{x\in \mathcal{C}(R)\} ,
    \end{align*}
    which proves~\eqref{following1}.
    Then, we obtain
    \begin{align*}
        &\check{\mathbb{P}}\left(X_{n+1} \in dy \mid \tilde{\mathcal{F}}_n,\delta_n=1\right)\\
        &= \check{\mathbb{P}}\left(\varphi_{n+1}\left(X_n,\delta_n,U^x_{n+1}\right) \in dy \mid \tilde{\mathcal{F}}_n,\delta_n=1\right)\\
        &=
        \check{\mathbb{P}}\left(\varphi_{n+1}\left(X_n,\delta_n,U^x_{n+1}\right) \in dy ,X_n\in \mathcal{C}(R)\mid \tilde{\mathcal{F}}_n,\delta_n=1\right)\\
        &=
        \check{\mathbb{P}}\left(\varphi^3(U^x_{n+1}) \in dy ,X_n\in \mathcal{C}(R)\mid \tilde{\mathcal{F}}_n,\delta_n=1\right)\\
        &=
        \mathbb{P}\left(\varphi^3(U^x_{n+1}) \in dy \right)\\
        &= \nu(dy) ,
    \end{align*}
    which proves~\eqref{following2}. When $x\in \mathcal{C}(R)$, we have
    \begin{align*}
        &\check{\mathbb{P}}\left(X_{n+1} \in dy \mid \tilde{\mathcal{F}}_n,X_n=x,\delta_n=0\right)\\
        &=\check{\mathbb{P}}\left(\varphi_{n+1}(X_n,\delta_n,U^x_{n+1}) \in dy ,X_n \in \mathcal{C}(R)\mid \tilde{\mathcal{F}}_n,X_n=x,\delta_n=0\right)\\
        &=\check{\mathbb{P}}\left(\varphi^2(X_n,U^x_{n+1}) \in dy ,X_n \in \mathcal{C}(R)\mid \tilde{\mathcal{F}}_n,X_n=x,\delta_n=0\right)\\
        &=\check{\mathbb{P}}\left(\varphi^2(x,U^x_{n+1}) \in dy \mid \tilde{\mathcal{F}}_n,X_n=x,\delta_n=0\right)\\
        &=\mathbb{P}\left(\varphi^2(x,U^x_{n+1}) \mid \tilde{\mathcal{F}}_n,X_n=x,\delta_n=0\right)\\
        &=
        \frac{P(n,n+1,dy)-\beta \nu(dy  )}{1-\beta} .
    \end{align*}
    Similarly, when $x\notin \mathcal{C}(R)$, we have
    $$\check{\mathbb{P}}\left(X_{n+1} \in dy \mid \tilde{\mathcal{F}}_n,X_n=x,\delta_n=0\right)=P(n,x,n+1,dy) ,$$
    which proves~\eqref{following3}.

    We now prove that $\left(\check{X}_n\right)_{n\in \mathbb{N}}=\left(X_n,\delta_n\right)_{n\in \mathbb{N}}$ is a time-inhomogeneous Markov chain with transition probability $\check{P}$ in~\eqref{construction1}--\eqref{construction3}. That is, 
    $$\check{\mathbb{P}}\left(X_{n+1} \in B, \delta_{n+1}=j \mid \tilde{\mathcal{F}}_n,X_n=x,\delta_n=i\right) = \check{P}\left(n,(x,i),n+1,B\times\{j\}\right) .$$
    When $i=1$ and $j=1$, we have
    \begin{align*}
        &\check{\mathbb{P}}\left(X_{n+1} \in B, \delta_{n+1}=1 \mid \tilde{\mathcal{F}}_n,X_n=x,\delta_n=1\right)\\
        &= \check{\mathbb{P}}\left(X_{n+1} \in B, X_{n+1}\in \mathcal{C}(R), U^{\delta}_{n+1}\le\beta \mid \tilde{\mathcal{F}}_n,X_n=x,\delta_n=1\right)\\
        &= \beta \check{\mathbb{P}}\left(\varphi^3_{n+1}(U^x_{n+1})\in B \cap \mathcal{C}(R)\right)\\
        &=\beta \nu\left(B\cap \mathcal{C}(R)\right)\\
        &=\nu^*(B_1)\\
        &=\check{P}(n,x_1,n+1,B_1) ,
    \end{align*}
    and when $i=1$ and $j=0$, we have:
    \begin{align*}
        &\check{\mathbb{P}}\left(X_{n+1} \in B, \delta_{n+1}=0 \mid \tilde{\mathcal{F}}_n,X_n=x,\delta_n=1\right)\\
        &=\check{\mathbb{P}}\left(X_{n+1} \in B, X_{n+1}\in \mathcal{C}(R), U^{\delta}_{n+1}>\beta \mid \tilde{\mathcal{F}}_n,X_n=x,\delta_n=1\right)\\
        &\quad+\check{\mathbb{P}}\left(X_{n+1} \in B, X_{n+1}\in \mathcal{C}(R)^c \mid \tilde{\mathcal{F}}_n,X_n=x,\delta_n=1\right)\\
        &= (1-\beta) \check{\mathbb{P}}\left(\varphi^3_{n+1}(U^x_{n+1}) \in B\cap \mathcal{C}(R)\right)+\check{\mathbb{P}}\left(\varphi^3_{n+1}(U^x_{n+1}) \in B\cap \mathcal{C}(R)^c\right)\\
        &=(1-\beta)\nu(C\cap \mathcal{C}(R))+\nu(C\cap \mathcal{C}(R)^c)\\
        &=\nu^*(B_0)\\
        &=\check{P}(n,x_1,n+1,B_0) .
    \end{align*}
     When $X_n=x \in \mathcal{C}(R)$, $i=0$ and $j=0$, we have 
    \begin{align*}
        &\check{\mathbb{P}}\left(X_{n+1} \in B, \delta_{n+1}=0 \mid \tilde{\mathcal{F}}_n,X_n=x,\delta_n=0\right)\\
        &=  \check{\mathbb{P}}\left(X_{n+1} \in B, X_{n+1}\in \mathcal{C}(R), U^{\delta}_{n+1}>\beta \mid \tilde{\mathcal{F}}_n,X_n=x,\delta_n=0\right)\\
        &\quad+\check{\mathbb{P}}\left(X_{n+1} \in B, X_{n+1}\in \mathcal{C}(R)^c \mid \tilde{\mathcal{F}}_n,X_n=x,\delta_n=0\right)\\
        &=(1-\beta)\check{\mathbb{P}}\left(\varphi^2_{n+1}(x,U^x_{n+1})\in B\cap \mathcal{C}(R)\right)
        + \check{\mathbb{P}}\left(\varphi^2_{n+1}(x,U^x_{n+1})\in B\cap \mathcal{C}(R)^c\right)\\
        &=
        \frac{1}{1-\beta}\left\{(1-\beta)P(n,x,n+1,B)+P(n,x,n+1,B\cap \mathcal{C}(R)^c)-\beta\left[(1-\beta)\nu(B)+\nu(B\cap \mathcal{C}(R)^c)\right]\right\}\\
        &=\check{P}(n,x_0,n+1,B_0) ,
    \end{align*}
    and when $X_n=x \in \mathcal{C}(R)$, $i=0$ and $j=1$ we have
    \begin{align*}
         &\check{\mathbb{P}}\left(X_{n+1} \in B, \delta_{n+1}=1 \mid \tilde{\mathcal{F}}_n,X_n=x,\delta_n=0\right)\\
         &=\check{\mathbb{P}}\left(X_{n+1} \in B, X_{n+1}\in \mathcal{C}(R), U^{\delta}_{n+1}\le\beta \mid \tilde{\mathcal{F}}_n,X_n=x,\delta_n=0\right)\\
         &=\beta \check{\mathbb{P}}\left(\varphi^2_{n+1}(B\cap \mathcal{C}(R))\right)\\
         &=\check{P}(n,x_0,n+1,B_1) .
    \end{align*}
    Finally, when $X_n=x\in\mathcal{C}(R)^c$ and $j=1$, we have
    \begin{align*}
        &\check{\mathbb{P}}\left(X_{n+1} \in B, \delta_{n+1}=1 \mid \tilde{\mathcal{F}}_n,X_n=x,\delta_n=0\right)\\
        &=\check{\mathbb{P}}\left(X_{n+1} \in B, X_{n+1}\in \mathcal{C}(R), U^{\delta}_{n+1}\le\beta  \mid \tilde{\mathcal{F}}_n,X_n=x,\delta_n=0\right)\\
        &=\beta \check{\mathbb{P}}\left(\varphi^1_{n+1}(B \cap \mathcal{C}(R))\right)\\
        &=\check{P}(n,x,n+1,B_1) ,
    \end{align*}
    and when $X_n=x\in\mathcal{C}(R)^c$ and $j=0$, we have
    \begin{align*}
          &\check{\mathbb{P}}\left(X_{n+1} \in B, \delta_{n+1}=0 \mid \tilde{\mathcal{F}}_n,X_n=x,\delta_n=0\right)\\
          &=\check{\mathbb{P}}\left(X_{n+1} \in B, X_{n+1}\in \mathcal{C}(R), U^{\delta}_{n+1}\le\beta  \mid \tilde{\mathcal{F}}_n,X_n=x,\delta_n=0\right)\\
          &\quad+\check{\mathbb{P}}\left(X_{n+1} \in B, X_{n+1}\in \mathcal{C}(R)^c \mid \tilde{\mathcal{F}}_n,X_n=x,\delta_n=0\right)\\
          &=\beta \check{\mathbb{P}}\left(\varphi^1_{n+1}(B\cap \mathcal{C}(R))\right)+\check{\mathbb{P}}\left(\varphi^1_{n+1}(B\cap \mathcal{C}(R))^c\right)\\
          &=\check{P}(n,x,n+1,B_0) .
    \end{align*}

    By the above argument, we obtain
    $$\check{\mathbb{P}}\left(\check{X}_{n+1}\in \cdot \mid \mathcal{F}_n,\delta_n=1\right)=\nu^*(\cdot),$$
    which implies the independence between the pre-$n$ and the post-$n$ processes given $\delta_n=1$.
\end{proof}

Let us prove that the time-inhomogeneous Markov chain $X$ given in~\eqref{MarkovChain} is identical in law to the marginal chain of the split chain.
\begin{Prop}
    The chain $(X_n)_{n \in \mathbb{Z}_+}$ is identical in law to the marginal chain $(X_n)_{n \in \mathbb{Z}_+}$ of the split chain $(\check{X}_n)_{n \in \mathbb{Z}_+}$: that is, for any initial distribution $\lambda$ on $\mathcal{B}(\mathcal{X})$ and any $E_1,\cdots,E_n \in \mathcal{B}(\mathcal{X}),\forall n\in \mathbb{N}$,
    \begin{equation}\label{eq:Marginal}
    \check{\mathbb{P}}_{\lambda^*}(X_1\in E_1,\cdots,X_n\in E_n)= \mathbb{P}_{\lambda}(X_1\in E_1,\cdots,X_n\in E_n).
    \end{equation}
\end{Prop}
\begin{proof}
    First we let $k=1$ and consider the case of the Dirac point mass $\lambda= \epsilon_x$.
    When $x\in \mathcal{C}(R)^c$, we have, by~\eqref{construction1},
    $$\int_{\check{\mathcal{X}}} \epsilon_x^*(dy_i)\check{P}(0,y_i,1, A_0 \cup A_1)=\check{P}(0,x_0,1, A_0 \cup A_1)= P(0,x,1,A_0\cup A_1)^*=P(0,x,1,A).$$
    On the other hand, when $x\in \mathcal{C}(R)$. We have, from $\eqref{construction2}$ and \eqref{construction3},
    \begin{align*}
        &\int_{\check{\mathcal{X}}} \epsilon_x^*(dy_i)\check{P}(0,y_i,1, A_0 \cup A_1)\\
        &=\beta\check{P}(0,x_1,1,A_0\cup A_1)+(1-\beta)\check{P}(0,x_0,1,A_0\cup A_1)\\
        &=\beta\nu(A)+(1-\beta)\frac{P(0,x,1,A_0\cup A_1)^*-\beta\nu^*(A_0\cup A_1)}{1-\beta}\\
        &=P(0,x,1,A).
    \end{align*}
    Thus we have, for any initial distribution $\lambda$,
    $$\int_{\mathcal{X}}\lambda(dx)P(0,x,1,A)= \int_{\mathcal{X}}\lambda^*(dy_i)\check{P}(0,y_i,1,A_0\cup A_1),$$
    which implies $\mathbb{P}_{\lambda}(X_1\in E_1)=\check{\mathbb{P}}_{\lambda^*}(X_1\in E_1)$ for any $E_1\in \mathcal{B}(\mathcal{X})$. 
    
    Suppose we have~\eqref{eq:Marginal} for $n$ and
    $E_1,\cdots,E_n \in \mathcal{B}(\mathcal{X})$. Let $f:\mathcal{X}\to \mathbb{R}$ be any bounded measurable function. On $\{X_n \in \mathcal{C}(R)\}$, by the tower property we have
        $$\check{\mathbb{E}}_{\lambda^*}[f(X_{n+1}) \mid X_1,\cdots,X_n]=\check{\mathbb{E}}_{\lambda^*}[\check{\mathbb{E}}_{\lambda^*}[f(X_{n+1}) \mid \mathcal{G}_n] \mid X_1,\cdots,X_n].$$
        By \eqref{following2}, \eqref{following3} we have
        \begin{align*}
        \check{\mathbb{E}}_{\lambda^*}[f(X_{n+1}) \mid \mathcal{G}_n;\delta_n=1]&=\int_{\mathcal{X}} f(y)\nu(dy),\\
        \check{\mathbb{E}}_{\lambda^*}[f(X_{n+1}) \mid \mathcal{G}_n;\delta_n=0]&= \int_{\mathcal{X}} f(y)\frac{P(n,X_n,n+1,dy)-\beta \nu(dy)}{1-\beta}.
        \end{align*}
        Hence,
        \begin{align*}
            \check{\mathbb{E}}_{\lambda^*}[f(X_{n+1}) \mid \mathcal{G}_n] &= \beta\check{\mathbb{E}}_{\lambda^*}[f(X_{n+1}) \mid \mathcal{G}_n;\delta_n=1]+(1-\beta)\check{\mathbb{E}}_{\lambda^*}[f(X_{n+1}) \mid \mathcal{G}_n;\delta_n=0]\\
            &=\int_{\mathcal{X}}P(n,X_n,n+1,dy)f(y).
        \end{align*}
        On $\{X_n \notin \mathcal{C}(R)\}$, according to \eqref{following1} and \eqref{following3} we have
        $$\check{\mathbb{E}}_{\lambda^*}[f(X_{n+1}) \mid \mathcal{G}_n]=\int_{\mathcal{X}}P(n,X_n,n+1,dy)f(y).$$
    Combining the above argument, we have that on the whole sample space,
    $$\check{\mathbb{E}}_{\lambda^*}[f(X_{n+1}) \mid \mathcal{G}_n]=\int_{\mathcal{X}}P(n,X_n,n+1,dy)f(y),$$
    which implies $$\check{\mathbb{E}}_{\lambda^*}[f(X_{n+1}) \mid X_1,\cdots,X_n]=\mathbb{E}_{\lambda}[f(X_{n+1}) \mid X_1,\cdots,X_n].$$
    Thus we obtain~\eqref{eq:Marginal} for $n+1$ and $E_1,\cdots,E_{n+1}\in \mathcal{B}(\mathcal{X})$.
\end{proof}

\subsection{Return to small set with geometric tail's probability}

We denote:
\begin{align*}
    K_n(x,\cdot)&:= P(n-1,x,n,\cdot).
\end{align*}
Then, Doeblin's Condition becomes 
$$K_n(x,\cdot)\geq \beta \nu(\cdot)$$
for all $n \in \mathbb{Z}_+$ and $x\in \mathcal{C}(R)$. 
By the Drift Condition, we have
\begin{align*}
    K_n V(x)&=\int_{\mathcal{X}}P(n-1,x,n,dy) V(y)\\
    &\leq C+\gamma V(x)\\
    &\leq \gamma V(x) + C'.
\end{align*}
where we denote $C':=\frac{C}{1-\gamma}$. We will see that the Drift Condition yields uniformly return to $\mathcal{C}(R)\times\{1\}$ with geometric tails. 

\begin{Lem}\label{GeoReturnOriginal} Suppose Assumption~\ref{DriftCondition} and Assumption~\ref{DoeblinCondition} hold. We know that $R>\frac{C'}{1-\gamma}$. Note $\rho:=\gamma^{}+\frac{C'}{R}<1$. Then for every starting time $s$ and every initial state $x$(i.e., $X_s=x$), we have
$$\mathbb{P}_{s,x}\left(\tau>n\right)\leq \frac{V(x)}{R}\rho^{n}, \forall  n\in \mathbb{Z_+}$$
where $\tau:=\inf \left\{k\in\mathbb{N}:X_{s+k}\in \mathcal{C}(R)\right\}$.
In particular, $\mathbb{P}_{s,x}\left(\tau=\infty\right)=0$.
\end{Lem}
\begin{proof}
    When $y\notin \mathcal{C}(R)$, we have $V(y) > \mathcal{C}(R)$. Then for any $t$, we have
    \begin{align*}
        K_tV(y)&\leq \gamma^{}V(y) + C' \\
        &= \left(\gamma^{{}}+\frac{C'}{V(y)}\right)V(y)\\
        &\leq \left(\gamma^{}+\frac{C'}{R}\right)V(y)\\
        &=\rho V(y).
    \end{align*}
    For any $n\geq 0$, on $\left\{\tau>n\right\}$ we have $X_{s+n}\notin \mathcal{C}(R)$, which implies $V(X_{s+n})>R$. Therefore, we have
    $$R\textbf{1}\left\{\tau>n\right\}\leq V(X_{s+n})\textbf{1}\left\{\tau>n\right\}.$$
    Taking expectation we have
    $$R\mathbb{P}_{s,x}\left(\tau>n\right)\leq \mathbb{E}_{s,x}\left[V(X_{s+n})\textbf{1}\left\{\tau>n\right\}\right].$$
    For $n=0$, we have
    $$\mathbb{E}_{s,x}\left[V(X_s)\textbf{1}\left\{\tau>0\right\}\right]= V(x)\textbf{1}\{x\notin \mathcal{C}(R)\}\leq V(x).$$
    Suppose for some $n\geq0$ we have
    $$\mathbb{E}_{s,x}\left[V(X_{s+n})\textbf{1}\left\{\tau>n\right\}\right]\leq \rho^n V(x).$$
    Then using the tower property and the fact that
    $$\textbf{1}\left\{\tau>n+1\right\}=\textbf{1}\left\{\tau>n\right\}\textbf{1}\left\{X_{s+n+1}\notin \mathcal{C}(R)\right\}\leq \textbf{1}\left\{\tau>n\right\},$$
    we obtain
    \begin{align*}
        &\,\,\,\,\,\,\,\,\mathbb{E}_{s,x}\left[V(X_{s+n+1})\textbf{1}\left\{\tau>n+1\right\}\right]\\
        &=\mathbb{E}_{s,x}\left[\textbf{1}\left\{\tau>n\right\}\mathbb{E}_{s,x}\left[V(X_{s+n+1})\textbf{1}\left\{X_{s+n+1}\notin \mathcal{C}(R)\right\} \mid \mathcal{F}_{n}^s\right]\right]\\
        &\leq \mathbb{E}_{s,x}\left[\textbf{1}\left\{\tau>n\right\}\mathbb{E}_{s,x}\left[V(X_{s+n+1}) \mid \mathcal{F}_{n}^s\right]\right]\\
        &=\mathbb{E}_{s,x}\left[\textbf{1}\left\{\tau>n\right\}K_{s+n+1}(X_{s+n})\right]  \\
        &\leq \mathbb{E}_{s,x}\left[\textbf{1}\left\{\tau>n\right\} \rho V(X_{s+n})\right] \\
        &= \rho\mathbb{E}_{s,x}\left[V(X_{s+n})\textbf{1}\left\{\tau>n\right\} \right]\\
        &\leq \rho \cdot \rho^nV(x)\\
        &=\rho^{n+1}V(x),
    \end{align*}
    where $\mathcal{F}_{n}^s:=\sigma(X_s,X_{s+1},\cdots, X_{s+n})$.
    Thus for all $n\geq0$, we have
    $$\mathbb{E}_{s,x}\left[V(X_{s+n})\textbf{1}\left\{\tau>n\right\}\right]\leq \rho^n V(x).$$
    Finally, we obtain
    $$\mathbb{P}_{s,x}\left(\tau>n\right)\leq \frac{1}{R}\mathbb{E}_{s,x}\left[V(X_{s+n})\textbf{1}\left\{\tau>n\right\}\right]\leq \frac{V(x)}{R}\rho^n.$$
\end{proof}

Now we lift such exponential tail of return time to the split chain.
\begin{Thm}\label{GeoReturnSplit}
    Suppose Assumption~\ref{DriftCondition} and Assumption~\ref{DoeblinCondition} hold. Then for every starting time $s$ and every initial state $x_i$(i.e., $\check{X}_s=x_i$) with $V(x)<\infty$, there exist constant $K>0$ and $0<\zeta<1$ such that
$$\check{\mathbb{P}}_{s,x_i}\left(\check{\tau}>n\right)\leq K \zeta^{n},\quad \text{for }\forall  n\in \mathbb{Z_+}$$
where $\check{\tau}:=\inf \left\{k\in\mathbb{Z}_+:\check{X}_{s+k}\in \mathcal{C}(R)\times\{1\}\right\}$.
In particular, $\check{\mathbb{P}}_{s,x_i}\left(\check{\tau}=\infty\right)=0$.
\end{Thm}
\begin{proof}
    Let $\sigma_0 := \tau$ and 
\[
\sigma_{k+1} := \inf\{t>\sigma_k : X_{s+t} \in \mathcal{C}(R)\}, \qquad k\ge0,
\]
denote the successive entrance times of the original chain into $\mathcal{C}(R)$.  
By Lemma~\ref{GeoReturnOriginal} and the strong Markov property,
there exist constants $C_1 < \infty$ and $\rho \in (0,1)$ which do not depend on $k$ such that
\[
\sup_{s,x}\,
\mathbb{P}_{s,x}\!\left(
\sigma_{k+1}-\sigma_k > n
\ \middle|\ 
\mathcal{F}_{\sigma_k}^s
\right)
\le C_1 \rho^{\,n},\qquad n\ge0.
\]
Hence there exist $\theta_0>0$ and $M<\infty$ such that
\begin{equation}\label{eq:mgf_bound}
\sup_{s,x}\,
\mathbb{E}_{s,x}\!\left[
e^{\theta_0(\sigma_{k+1}-\sigma_k)}
\ \middle|\ 
\mathcal{F}_{\sigma_k}^s
\right]
\le M.
\end{equation}

At each visit time $\sigma_k$, the split construction tosses an independent
$\mathrm{Bernoulli}(\beta)$ coin (conditionally on the past): with probability
$\beta$ the chain regenerates at layer~1, otherwise it stays in layer~0.  
Let
\[
G := \inf\{k\ge0 : \delta_{s+\sigma_k}=1\},
\]
so that $\check{\tau} = \sigma_G$.
Given the past, $G$ is geometric$(\beta)$ and independent of the inter-visit
increments.

Conditioning on $G$ and using~\eqref{eq:mgf_bound},
\begin{align*}
\check{\mathbb{E}}_{s,x_i}\!\left[
e^{\theta \check{\tau}}
\right]
&=
\check{\mathbb{E}}_{s,x_i}\!\left[
e^{\theta\sigma_0}
\prod_{j=0}^{G-1}
e^{\theta(\sigma_{j+1}-\sigma_j)}
\right]\\
&=\sum_{k=0}^{\infty}\check{\mathbb{E}}_{s,x_i}\!\left[e^{\theta\sigma_0}
\left(\prod_{j=0}^{k-1}
e^{\theta(\sigma_{j+1}-\sigma_j)}\right)\textbf{1}\{G=k\}\right]\\
&=\sum_{k=0}^{\infty}\check{\mathbb{E}}_{s,x_i}\!\left[e^{\theta\sigma_0}
W_{k-1}\right]\quad\left(\text{where }W_k:=\prod_{j=0}^{k}
e^{\theta(\sigma_{j+1}-\sigma_j)}\textbf{1}\{\delta_{s+\sigma_j}=0\}\right)\\
&=\sum_{k=0}^{\infty}\check{\mathbb{E}}_{s,x_i}\left[e^{\theta\sigma_0}W_{k-1}\right]\cdot \beta\\
&=\sum_{k=0}^{\infty}\check{\mathbb{E}}_{s,x_i}\left[e^{\theta\sigma_0}W_{k-2}\cdot \check{\mathbb{E}}_{s,x_i}\left[e^{\theta (\sigma_k -\sigma_{k-1})} \mid \mathcal{F}_{\sigma_{k-1}}^s\right]\right]\cdot (1-\beta)\beta\\
&\leq \sum_{k=0}^{\infty}\check{\mathbb{E}}_{s,x_i}\left[e^{\theta\sigma_0}W_{k-2}\right]\cdot M(1-\beta)\beta\\
&\leq \cdots \\
&\leq \sum_{k=0}^{\infty}\check{\mathbb{E}}_{s,x_i}[e^{\theta \sigma_0}]M^k(1-\beta)^k \beta .
\end{align*}

\noindent From Lemma~\ref{GeoReturnOriginal}, $\sigma_0$ has geometric tail if $V(x)<\infty$. Because $M(\theta)$ is continuous at $\theta=0$ with $M(0)=1$, we can choose
$0<\theta\le\theta_0$ so small that $(1-\beta)M<1$ and $\check{\mathbb{E}}_{s,x_i}[e^{\theta \sigma_0}] < \infty$.
Then the geometric series converges and
\[
\sup_{s,x_i}
\check{\mathbb{E}}_{s,x_i}\!\left[e^{\theta\check{\tau}}\right]
\le
\frac{\check{\mathbb{E}}_{s,x_i}[e^{\theta \sigma_0}]\beta}{1-(1-\beta)M}
<\infty.
\]
Thus, for all $n\ge0$,
\[
\check{\mathbb{P}}_{s,x_i}\!\left(
\check{\tau} > n
\right)
\le\check{\mathbb{E}}_{s,x_i}[e^{\theta \check{\tau}}]\cdot e^{-\theta n}
\le
\frac{\check{\mathbb{E}}_{s,x_i}[e^{\theta \sigma_0}]\beta}{1-(1-\beta)M}\, e^{-\theta n}
=:
K\,\zeta^{\,n},
\quad \zeta := e^{-\theta}\in(0,1).
\]
Letting $n\to\infty$, we have
$\check{\mathbb{P}}_{s,x_i}(\check{\tau}=\infty)=0$.
\end{proof}

\subsection{The proof of SLLN}

We let $\tau_0$ denote the first entrance time of the split chain to the set $\mathcal{C}(R)\times \{1\}$, and $\tau_k$ the $k^{\text{th}}$ entrance time to $\mathcal{C}(R)\times \{1\}$ subsequent to $\tau_0$. These random variables are defined inductively as 
\begin{align*}
    \tau_0&=\min (n\ge0: \delta_n=1),\\
    \tau_k &=\min(n > \tau_{k-1}:\delta_n=1), \quad\text{for }k\ge 1.
\end{align*}

With all the arguments so far, we obtain the sequence of independent, but not identically distributed regeneration cycles $(C_l)_{l\in \mathbb{Z}_+}$ almost surely, where 
$$C_l:=\{\tau_l+1,\tau_l+2,\cdots,\tau_{l+1}\}.$$
For each $k$ define
$$D^0_k(g):=\sum_{j=\tau_k +1}^{\tau_{k+1}}g(X_j),$$
where $g:\mathcal{X}\to\mathbb{R}$ is a bounded function. From Property~\ref{Property1}, every $X_{\tau_k+1}$ is regenerated by drawing $\nu$ and so independent of the past, we have that $(D^0_k)_{k\in \mathbb{Z}_+}$ is an independent random process.

Let $N(n)$ be the number of regenerations up to time $n$:
$$N(n):=\max\{k:\tau_k \le n\}.$$
Then we have
$$n= \tau_0+\sum_{l=0}^{N(n)-1}L_l + r(n),$$
where $L_l:=\tau_{l+1}-\tau_l$ is the length of each regeneration cycle $C_l$ and $r(n)$ is the remaining term. From Theorem~\ref{GeoReturnSplit}, We have the following result.
\begin{Cor}\label{UpToInfty}
For any initial state $x_i$ with $V(x)<\infty$, we have $N(n)\underset{a.s.}{\uparrow}\infty$.
\end{Cor}
\begin{proof}
    By Theorem 4.5, we have $\check{\mathbb{P}}_{0,x_i}(\tau_0<\infty)=1$. Thus we can consider $\tau_0$ to be the initial time by the strong Markov property and modify $N(n)=\max\{k:\tau_k\leq \tau_0+n\}$ to prove the original result. For any $M\in \mathbb{N}$, we have
    \begin{align*}
        \check{\mathbb{P}}_{\tau_0}\{N(n) \le M\}&= \check{\mathbb{P}}_{\tau_0}\{\tau_M 
        \ge n\}\\
        &\le \check{\mathbb{P}}_{\tau_0}\left\{1\le \exists l\le M\,\,s.t.  \,\,\tau_{l}-\tau_{l-1}\ge \frac{n}{M}\,\,\right\} \\
        &\le M\check{\mathbb{P}}_{\nu^*}\left\{\check{\tau}>\left\lfloor \frac{n}{M} \right\rfloor\right\}\\
        &\le MK \zeta^{\left\lfloor \frac{n}{M} \right\rfloor}\xrightarrow{n\to\infty}0.
    \end{align*}
    Thus the proof is complete.
\end{proof}

 We will see that the length of each cycle has uniform twice moment.
\begin{Lem}\label{TwiceMoment} Under the condition of Theorem~\ref{GeoReturnSplit}, we have $\sup\limits_{l\in\mathbb{Z}_+}\check{\mathbb{E}}[L_l^2]< \infty$ and $\check{\mathbb{E}}[\tau_0^2]<\infty$.
\end{Lem}
\begin{proof}Since
    \begin{align*}
        \check{\mathbb{P}}(L_l>m) &= \check{\mathbb{P}}_{\nu^*}(\check{\tau}>m)
        \le K \zeta^m,
    \end{align*}
    we have
    \begin{align*}
        \check{\mathbb{E}}[L_l^2]&=\sum_{k=0}^{\infty} \check{\mathbb{P}}\{L^2_l > k\}\\
        &=\sum_{m=0}^{\infty} (2m+1)\check{\mathbb{P}}\{L_l >m\}\\
        &\le \sum_{m=0}^{\infty} (2m+1)K\zeta^m,
    \end{align*}
    which implies $\sup\limits_{l\in\mathbb{Z}_+}\check{\mathbb{E}}[L_l^2]< \infty$. By a similar argument, we have $\check{\mathbb{E}}[\tau_0^2]<\infty$.
\end{proof}

The LLN for independent but not identically distributed random variables plays a key role; see, e.g. Corollary 5.22 of Kallenberg~\cite{Kallenberg}.
\begin{Lem}
Assume that $X_1,X_2,
\cdots$ are independent with means $\mu_1,\mu_2,\cdots$ and variances $\sigma_1^2,\sigma_2^2,\cdots$ such that $\sum\limits_{k=1}^{\infty}\frac{\sigma_k^2}{k^2}<\infty$ called Kolmogorov's criterion. Then
$$\frac{X_1+\cdots+X_n-(\mu_1+\cdots+ \mu_n)}{n}\overset{a.s.}{\longrightarrow} 0.$$
\end{Lem}

We now obtain the SLLN in the equilibrium case.

\begin{Thm}\label{SLLN1}
Suppose Assumptions~\ref{DriftCondition} and ~\ref{DoeblinCondition} hold. Let $x\in \mathcal{X}$ be such that $V(x) <\infty$. Let $g:\mathcal{X}\to \mathbb{R}$ be a bounded measurable function. Then, for the time-inhomogeneous Markov chain $(X_n)_{n\in \mathbb{Z}_+}$ with $X_0\sim \mu_0$, we have
$$\frac{1}{n}\sum_{k=0}^{n-1}g(X_k)-\frac{1}{n}\sum_{k=0}^{n-1}\mu_k(g)\xrightarrow[n\to \infty]{\text{a.s.}}0.$$
\end{Thm}
\begin{proof}
We denote:
$$S_n=\sum_{i=0}^{n}g(X_i)-\sum_{i=0}^{n}\mu_i(g)=\sum_{i=0}^{\tau_0}\left(g(X_i)-\mu_i(g)\right)+\sum\limits_{l=0}^{N(n)-1}D_l+R(n), $$
where $$D_l:=D^0_l-\sum_{j=\tau_l +1}^{\tau_{l+1}}\mu_j(g),$$
and $R(n)$ is the remaining term. Then we have
$$\frac{S_n}{n+1}= \frac{\sum\limits_{i=0}^{\tau_0}g(X_i)+\sum\limits_{l=0}^{N(n)-1}D_l+R(n)}{n+1}=\frac{\frac{1}{N(n)}\sum\limits_{i=0}^{\tau_0}g(X_i)+\frac{1}{N(n)}\sum\limits_{l=0}^{N(n)-1}D_l+\frac{R(n)}{N(n)}}{\frac{n+1}{N(n)}}.$$
From Lemma~\ref{TwiceMoment} and that $g$ is bounded, we have $$\sup\limits_{l\in\mathbb{Z}_+} \check{\mathbb{E}}[D_l^2]\leq\sup\limits_{l\in\mathbb{Z}_+}\|g\|^2_{\infty}\check{\mathbb{E}}[L_l^2]<\infty$$which implies $\sup\limits_{l\in\mathbb{Z}_+}Var(D_l)<\infty$. Thus $(D_l)_{l\in \mathbb{Z}_+}$ satisfy Kolmogorov's criterion and we obtain
$$\frac{1}{N(n)}\sum\limits_{l=0 }^{N(n)-1}D_l\overset{a.s.}{\longrightarrow}0.$$
We claim that $\frac{R(n)}{N(n)}\overset{a.s.}{\longrightarrow}0$. Since for any $\varepsilon>0$, by Chebyshev's inequality, we have
$$\check{\mathbb{P}}(L_l>\varepsilon l)
\leq \frac{\sup _{l\in \mathbb{Z}_+}\check{\mathbb{E}}[L^2_l]}{\varepsilon^2 l^2},$$
which implies
$$\sum^{\infty}_{l=0}\check{\mathbb{P}}(L_l>\varepsilon l)<\infty.$$
By Borel-Cantelli lemma we have
$$\frac{L_l}{l}\overset{a.s.}{\longrightarrow}0.$$
Since the fact $r(n)<L_{N(n)}$, we have
$$\frac{r(n)}{N(n)}<\frac{L_{N(n)}}{N(n)}\overset{a.s.}{\longrightarrow}0,$$
which implies
$$\frac{R(n)}{N(n)}\leq\|g\|_{\infty}\frac{r(n)}{N(n)}\overset{a.s.}{\longrightarrow}0.$$
Since $\check{\mathbb{P}}(\tau_0=\infty)=0$ by Lemma~\ref{TwiceMoment}, we obtain
$$\limsup_{n\to\infty}\frac{S_n}{n+1}\leq \lim_{n\to\infty}\frac{1}{N(n)}\sum_{i=0}^{\tau_0}g(X_i) + \lim\limits_{n\to\infty}\frac{1}{N(n)}\sum\limits_{j=0}^{N(n)-1}D_j+\lim\limits_{n\to \infty}\frac{R(n)}{N(n)} =0.$$
The proof is complete.
\end{proof}

We now proceed to the proof of Theorem~\ref{generalLLN} as a consequence of Theorem~\ref{SLLN1} by using the coupling method based on the exponential ergodicity.

\begin{proof}[Proof of Theorem~\ref{generalLLN}]
    From the exponential ergodicity of Theorem~\ref{InvariantExistence} and by Goldstein's theorem (see Theorem 14.10 of Lindvall~\cite{Lindvall}), there exists a maximal coupling of two chains: 
    $$X^{(x)} \text{ with } X^{(x)}_0=x,\qquad X^{(\mu)} \text{ with } X^{(\mu)}_0\sim \mu_0$$
    such that the coupling time:
    $$T:=\inf\{N\in \mathbb{Z}_+:X_n^{(x)}=X_n^{(\mu)}\quad\text{for }n\ge N\}$$
    is almost surely finite. Consequently, for any bounded function $g:\mathcal{X} \to \mathbb{R}$, we have
    $$\left|\frac{1}{n} \sum_{k=0}^{n-1} \left(g(X^{(x)}_k)-g(X^{(\mu)}_k)\right)\right|\xrightarrow[n\to \infty]{a.s.}0,$$
    and combining with the SLLN under $X_0\sim \mu_0$, we have
    $$   \frac{1}{n}\sum_{k=0}^{n-1} g(X_k)-\frac{1}{n}\sum_{k=0}^{n-1} \mu_k(g)
    \;\xrightarrow[n \to \infty]{\ \text{a.s.}\ }\; 0 ,$$
    which complete the proof.
\end{proof}

\section*{Acknowledgments}
A. Lau is grateful to the Graduate School of Science, The University of Osaka, for the scholarship for international students. This research was supported by ISM 2025-ISMCRP-5007. The research of K. Yano was supported by JSPS KAKENHI grant no.'s JP24K06781, JP24K00526 and JP21H01002, and by JSPS Open Partnership Joint Research Projects grant no. JPJSBP120249936.


\begin{thebibliography}{99}

\bibitem{Doeblin}
V.~Bansaye, B.~Cloez, and P.~Gabriel.
\newblock Ergodic behavior of non-conservative semigroups via generalized
Doeblin's Conditions.
\newblock \emph{Acta Appl Math} \textbf{166}, 29--72, 2020.

\bibitem{Martin}
M. Hairer.
\newblock {\em Ergodic Properties of Markov Processes}.
\newblock Lecture Notes, July 2018.

\bibitem{HallHeyde1980}
P.~Hall and C.~C.~Heyde.
\newblock {\em Martingale Limit Theory and Its Application}.
\newblock Academic Press, New York, 1980.

\bibitem{Kallenberg}
O.~Kallenberg.
\newblock {\em Foundations of Modern Probability}, 3rd ed.
\newblock Springer, New York, 2021.

\bibitem{KemenySnell1976}
J.~G.~Kemeny and J.~L.~Snell.
\newblock {\em Finite Markov Chains}.
\newblock Springer, New York, 1976.

\bibitem{Lindvall}
T.~Lindvall.
\newblock {\em Lectures on the Coupling Method}.
\newblock Dover Publications, 2002.   

\bibitem{LiuLu2025}
Z.~Liu and D.~Lu.
\newblock Ergodicity of inhomogeneous Markov processes under general criteria.
\newblock {\em Frontiers of Mathematics in China}, to appear (2025).
\newblock DOI: 10.1007/s11464-023-0102-1.

\bibitem{MeynTweedie2009}
S.~P.~Meyn and R.~L.~Tweedie.
\newblock {\em Markov Chains and Stochastic Stability}, 2nd ed.
\newblock Cambridge University Press, 2009.

\bibitem{Nummelin}
E.~Nummelin.
\newblock A splitting technique for Harris recurrent Markov chains.
\newblock {\em Z. Wahrscheinlichkeitstheorie verw. Gebiete}, 43:309--318, 1978.

\bibitem{Vassiliou}
P.-C.G. Vassiliou.
\newblock Law of Large Numbers for Non-Homogeneous Markov Systems.
\newblock \emph{Methodol Comput Appl Probab} \textbf{22}, 1631--1658,
2020.

\end{thebibliography}
\end{document}